\documentclass[a4paper,11pt,leqno]{amsart}

\usepackage{graphics}
\usepackage{thmtools}
\usepackage{wasysym}

\usepackage[T1]{fontenc}    

\usepackage{amsthm}
\usepackage{amsbsy,amsmath,amssymb,amscd,amsfonts}
\usepackage[pagebackref=true]{hyperref}
\usepackage{url}            
\usepackage{booktabs}       
\usepackage{nicefrac}       
\usepackage{microtype}      

\usepackage{graphicx,float,latexsym,color}
\usepackage[font={scriptsize,it}]{caption}
\usepackage{subcaption}

\usepackage{makecell}

\usepackage[dvipsnames]{xcolor}
\newtheorem{alg}{Construction}

\newtheorem{proposition}{Proposition}

\newtheorem{corollary}{Corollary}
\newtheorem{lemma}{Lemma}
\theoremstyle{remark}
\newtheorem{remark}{Remark}
\theoremstyle{definition}
\newtheorem{definition}{Definition}

\hypersetup{
    pdftoolbar=true,        
    pdfmenubar=true,        
    pdffitwindow=false,     
    pdfstartview={FitH},    
    colorlinks=true,       
    linkcolor=OliveGreen,          
    citecolor=blue,        
    filecolor=black,      
    urlcolor=red           
}

\usepackage{lineno}

\usepackage{comment}
\usepackage{microtype}

\begin{document}

\title{Poristic virtues of a negative-pedal curve}


\author[L. G. Gheorghe]{Liliana Gabriela Gheorghe}



\maketitle













\textbf{Abstract.}

\small{We describe all  triangles that shares either circumcircle and pedal circle or  circumcircle and negative-pedal circle. Neither of these  pairs is  poristic; nevertheless,  the negative-pedal curve of the pedal-circle is a  (very) special i-conic that  points toward a poristic solution. Subsequently, other poristic pairs show up and the choreography swiftly begins.
 }
\bigskip

\bigskip

\bigskip

\section{Introduction}

\bigskip
Which triangles  share  circumcircle and pedal-circle? Which one share   circumcircle and  negative-pedal circle? Are these triangles related? Is it possible to
 draw  all of them?
 
 In this paper we give a poristic answer to this problem and provide functorial recipes 
to construct all these triangles.

The ground-case  obtains when pedal-point is the i-center:  
to find all triangles that share  in-circle and circumcircle. Of course, nowadays anyone knows that if two circles are, respectively the circumcircle and the i-circle of a reference triangle, then
they are so for infinitely many other triangles; as a mater a fact, any point of the circumcircle is a
vertex of one (and only one) triangle circumscribed to the i-circle. But till midst 1700, this phenomenon was not so well captured.
The algebraic relation between the radius of the i-circle and  circumcircle of a given triangle was proved  by Chapple in 1746 (see e.g. [OW] or [W]),
enforced the fact that not any two circles are meant to be i-circle and circumcirlce of some triangle.

 Nevertheless,  the first who understood the poristic nature 
 of this formula was
  Collin MacLaurin (1698-1746), a scotish mathematician, who proved a  special case of  Poncelet's porism: a porism for triangles and a pair of conics. For a self-contained proof of MacLaurin theorem, which use systems of triangles auto-polar w.r to a conic, see [GSO], section 9.5. An elementary  geometric proof of Poncelet porism, in its full generality is in [A]; see [P] for Poncelet's original proof.

When pedal-point is the orthocenter, the pedal circle is the (classic)
Euler circle (or the nine-point circle); as a mater a fact its proprieties were first proved by Poncelet and Brianchon in [BP] in relation with a problem of construction of a i-conic (the Apollonius hyperbola)!
The problem of finding all triangles that find all triangles that
shares the Euler-circle and circumcircle, was studied, with different methods in  [Ox],  [We] and quite recently  in
  [Pa],  who answered these questions in the acute case; in [G], the obtuse case is solved.
The approach we adopt here 
differs from those cited above.
The key observation is that the i-conic of a triangle, which have a prescribed focus into pedal point $D$ is  precisely the negative-pedal (curve)

--------------------------------------

\textbf{Keywords and phrases: } Euler circle, pedal circle, Poncelet's porism, negative-pedal curve, dual curve, polar reciprocal, circle inversion. 

\textbf{(2020)Mathematics Subject Classification: }51M04, 51M15, 51A05.

 of the pedal-circle w.r. to $D.$ 
This fact enables a poristic approach  that embedded
 a recipe for the construction of all poristic triangles, as well.
 
The  straightforwardness of the  proofs is  due to
a systematic use of inversive methods, polar duality and above all, to poristic virtues of the negative-pedal curve and to Poncelet's porism.

The reader not acquainted with inversive methods (circle inversion, dual curves, negative polar curve), may  consult either the  classic [P],[S],[Ch1],[Ch], or the beautiful books [A], [GSO]. For a very quick review, see [W] and the references therein.

\begin{figure}
\centering
\includegraphics[trim=100 50 200 0,clip,width=1.0\textwidth]{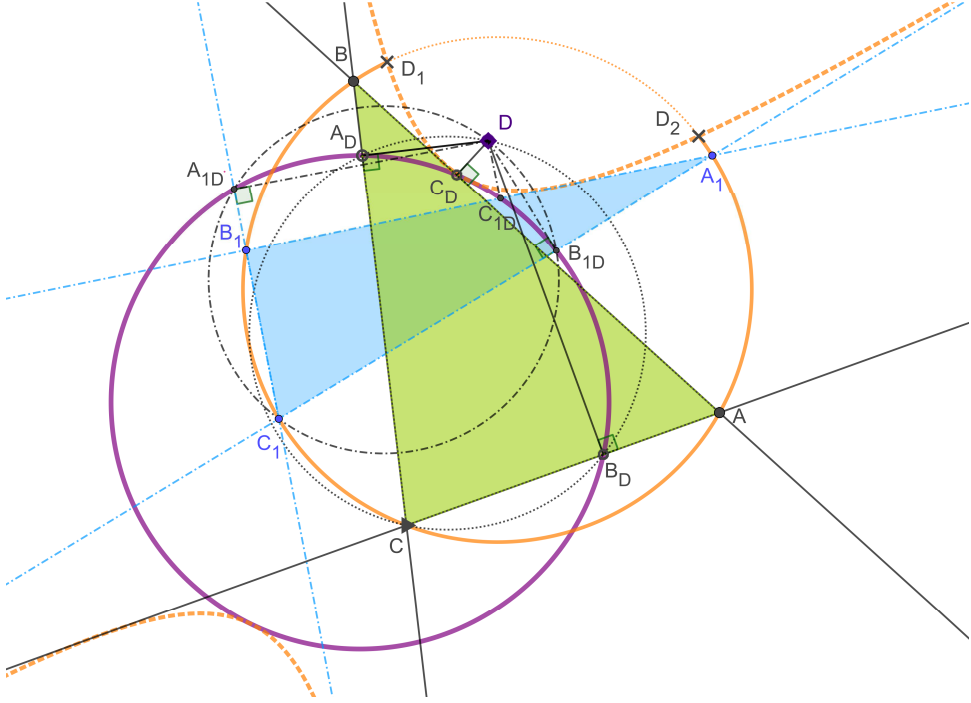}
\caption{ If
${\Gamma}$ (orange) is any circumconic of $\triangle{ABC}$, $\mathcal{E}_D$ the pedal-circle (purple) w.r. to $D$ and $\gamma_D$, (doted orange hyperbola) is the negative-pedal of  $\mathcal{E}_D$, then  $({\Gamma}, \gamma_D)$
form a poristic pair for $n=3.$
All triangles (e.g. green and blue) inscribed in $\Gamma$ and circumscribed to $\gamma_D$
share the same pedal-circle $\mathcal{E}_D.$
The arc $D_1D_2$ of $\Gamma,$ situated inside the conic $\gamma_D$ is infertile: it cannot contain vertices of such triangles, since there is no tangent line to $\gamma_D$ passing through these points.
}
\label{fig:2041Bis}
\end{figure}


\section{A pedal porism}

Let $\mathcal{T}$ a triangle, which we shall call reference triangle and let  a point $D$ neither on its sides, nor on its circumcircle, which we shall call  pedal-point.
\begin{definition}
  The triangle whose vertices are the projections of the pedal point $D$ on the sides of $\mathcal{T}$
     is the pedal triangle;  we call its circumcircle, $\mathcal{E_D},$  the
    pedal-circle.
    \label{defi:pedal-circle}
    \end{definition}
This definitions naturally extend those of classic Euler circle:
 if the pedal point $D$ is either the circumcenter, or the orthocenter, then pedal-circle is the (classic) Euler circle (or the nine-point circle).
Pedal-circles w.r. to i-centers  are simply i-circles (inscribed or exinscribed circle). For further notable pedal-circles and a comprehensive study of their proprieties,  see [PS].

 The following definition are classic.
 \begin{definition}
 The negative-pedal of a curve $\gamma$, w.r. to a pedal point $D$ 
  is the envelope of the perpendiculars through point  $P\in \gamma$, to line  $PD,$ as point $P$ sweeps  $\gamma$.
  We denoted it by  $\mathcal{N}(\gamma)$.
 \label{defi:negative_pedal}
 \end{definition}
 Negative pedal curves were 
  studied with some intensity by the end of the 1800; a interested reader might  see e.g. [Am1],[Am2]; for a quick acquaintance, see also [W].
 The following are two equivalent definitions of a polar dual of a curve.
 \begin{definition}
 The polar dual (or reciprocal) of a (regular) curve $\gamma$ w.r. to an inversion circle is the envelope of the polars of points $P,$ as $P$ sweeps $\gamma.$
 \end{definition}
  \begin{definition}
 The polar dual (or reciprocal) of a (regular) curve $\gamma$ w.r. to an inversion circle, denoted by $\mathcal{R}(\gamma)$ is the loci of the poles of the tangents $t_P$ to $\gamma,$ as  $P$ sweeps $\gamma.$
 \end{definition}
 Dual curves and the method of polar reciprocals are due to Poncelet; see 
 [P]; see also  [Ch], [Ch1]; see also [S].
 The following useful description of the negative-pedal, as a loci of points is classic.
 \begin{proposition}
 The negative-pedal of a curve $\gamma$, denoted by $\mathcal{N}(\gamma)$, 
 is the polar dual of its inverse: 
 $\mathcal{N}(\gamma)=\mathcal{R}(\gamma'),$ where $\gamma'$ is the inverse of $\gamma$ w.r. to  an inversion circle centred at the pedal point.
 \label{prop:negative_pedal_inverse}
 \end{proposition}

 Using Proposition ~\ref{prop:negative_pedal_inverse} and known facts on polar-dual of a circle, we get the following key result
 (see e.g. [La] and [Lo] for alternative approach)
 
 \begin{proposition}
 The negative-pedal of a circle 
 w.r. to  pedal point $D$ not on the refereed circle, is the
 conic centred at the center of the circle, whose focus is $D$ and
 whose main axis is (precisely) its  diameter  
  through $D.$ 
 
 The negative-pedal of a circle is either an ellipse 
 (when the pedal point is inside the circle) or a hyperbola
 (when $D$ is outside).
 If $D$ is on the circle, the pedal curve reduces to a point. The negative-pedal curve of a circle is never a parabola.
 \label{prop:circulo_pedal_negativo}
 \end{proposition}

 Before proceed, 
 let us  point out the key-fact, that ties-up i-conics,
pedal-circles and negative-pedal curves.

\begin{proposition} 
The i-conic  of $\mathcal{T}$ focused in $D$  
is the negative-pedal (curve) of its pedal-circle $\mathcal{E}_D$.
\label{proposition:Euler_circle_as_a_negative_pedal}
   \end{proposition}
   \begin{proof}
   The proof is a direct consequence of 
   the definition of the negative-pedal and pedal-circle 
   \end{proof}
   For a proof, see e.g. Chapter  7.5 [GSO]; for
   a nice construction of an i-conic  with a prescribed focus, see e.g.  [Ch1]; see also [GSO], Example 7.2.2.

   Now we may prove the key-ingredient of a  poristic approach.

  \begin{lemma} 
 Let $\mathcal {E}$  any circle, $D$ any point not on $\mathcal{E}$ and  $\gamma_D$ the negative-pedal of  $\mathcal {E}.$
 
  Then a  triangle have  pedal-circle $\mathcal{E}$ if and only if is circumscribed to $\gamma_D$ (its sides
 \footnote{ by "side" we mean the line that pass through two vertices of a triangle}
 tangents  the conic $\gamma_D$).
 \label{lemma:pedal_negativo}
 \end{lemma}
 \begin{proof}
 
Refer to figure~\ref{fig:2041Bis}

$\Rightarrow$ If $\triangle{ABC}$ have pedal-circle $\mathcal {E}$, then necessarily the feet of the perpendiculars from $D$ to the sides of this triangle are on $\mathcal {E}.$ 
Thus, by the definition of a negative-pedal curve as an envelope of lines (see ~\ref{defi:negative_pedal})
these sides  tangents the negative-pedal of circle $\mathcal {E},$ i.e.
the conic $\gamma_D.$

$\Leftarrow$
 By hypothesis, the conic $\gamma_D$ has a focus in $D$ and is inscribed in $\triangle{ABC};$
the feet of the
perpendiculars from $D$ to the sides of triangle determines the pedal-circle $\mathcal{E}_D$ of $\triangle{ABC}.$ On the other side, these feet of the perpendicular from the focus of the conic to the tangents to that conic,  
belong to one and the same circle,
$\mathcal{E}$, (the circle) whose negative-pedal is $\gamma_D$.
Thus, the pedal-circle $\mathcal{E}_D$ of $\triangle{ABC}$
is precisely $\mathcal{E}.$

\end{proof}
In other words, three arbitary  tangents  to a given conic $\gamma_D$ determine triangles whose  pedal-circle is  the (unique) circle which diameter is the main axis of the conic $\gamma_D.$

  This led to a first poristic result.
 \begin{proposition} (the pedal porism)
      Let  $\gamma_D$ the i-conic of $\mathcal{T}$ focused in $D$; then $\big(\mathcal{C}, \gamma_D \big)$ form a poristic pair.
      
        All triangles inscribed in $\mathcal{C}$ share the same pedal-circle $\mathcal{E}_D$
        if and only if are circumscribed to $\gamma_D.$
        \label{prop:poristica}
  \end{proposition} 
  And now the construction of all these triangles.
 \begin{alg}
 Refer to figure ~\ref{fig:2041Bis}.
 Let $C$ a point  located on $\mathcal{C}.$ Let
the circle of diameter $[CD]$ intercept ${\mathcal E}_D$ in
$A_D$ and $B_D.$
 The lines  $CA_D$ and $CB_D$ intercept (again) 
 $\mathcal{C}$ in $B$ and $A.$
 
 Then the line $AB$ is a tangent to $\gamma_D$ and the feet of $D$ over $AB$ is on circle $\mathcal{E}_D.$ All poristic triangles are obtainable in this manner.
 \end{alg}
 \begin{proof}
 By construction, $A_D$ and $B_D$ are 
  the feet of the perpendiculars from $D$ to $CA_D$ and $CB_D,$
 respectively. Since $\mathcal{E}_D$ is the negative-pedal of $\gamma_D,$ $CA_D$ and $CB_D$ are  two tangents to  $\gamma_D.$
 By Poncelet's  (MacLauren) porism, 
  $AB$ is  tangent to $\gamma_D$  since  $\big(\mathcal{C},\gamma_D\big)$
 form a poristic pair. The second assertion is now a consequence of the fact that $\mathcal{E}_D$ is the negative-pedal of $\gamma_D.$
 \end{proof}
 
 \begin{remark}
The proof above give an insight on infertile arcs of $\mathcal{C}.$ A necessary condition for this construction to work is the existence of tangents from $C$ to the i-conic $\gamma_D;$ and this happens iff $C$ is located on $\mathcal{C}$ and outside the i-conic (not on the  arc $D_1D_2$; see again figure ~\ref{fig:2041Bis}):
the  arc $D_1D_2$
 delimited by the intersection of the circumcircle with the i-conic focused in $O$ contain no vertices of admissible triangle.
\end{remark}
 
 \begin{remark}
  This construction still hold when the pedal point is $H$, which corresponds to classic Euler circles.
 This construction is different from those in [G] and also makes clear what happens in the obtuse case of an Euler circle. In this case, the pedal point $D$ lie outside the triangle yet inside the circumcircle.
 The i-conic $\gamma_D$ is a hyperbola and 
 this causes the infertile arcs.
 When the triangle is acute, its orthocenter is inside the triangle and the i-conic is an ellipses. In this case, the i-conic is inside the circumcircle, and any initial point $C$ is admissible.
 
 \end{remark}

\section{A polar porism}

In this  section,   we shall provide  another  construction of triangles sharing the same circumcircle and Euler negative-pedal circle, based on an ad-hoc polar-porism.

When not specified otherwise, 
we subtend that poles and polars, as well as the dual polars of curves or inversion are w.r. to inversion circle ${\mathcal{I}}$
centred at   $D.$

\begin{proposition}(a polar-porism)
Let $\mathcal{T}$ a  triangle and $\mathcal{C}$ its circumcircle 
and let an inversion circle $\mathcal{I}$ centered at a point $D,$ located neither on the sides of $\mathcal{T},$ nor on its circumcircle. The polars of the vertices of  $\mathcal{T}$,  w.r. to $\mathcal{I}$, determines a new  triangle $\mathcal{T}_p$;
let $\mathcal{C}_p$ its circumcircle.
Let $\gamma_D$ the i-conic of $\mathcal{T}$ focused in $D$

Then $\big[\mathcal{C}_p, \mathcal{C}\big]$ form a 
polar-poristic pair of circles, in the following sense: if from  (any) point $A$ of circle $\mathcal{C}$ we  let 
the polar of $A$ intercept circle $\mathcal{C}_p$ 
in two distinct points $B_p$ and $C_p,$ and subsequently the polars of 
$B_p$ and $C_p$ intercept $\mathcal{C}$ in $A$, $B$ and $A$, $C$, respectively.
 
 Finally, let $A_p$ be the pole of $BC;$ then
 
 i) $\triangle{ABC}$ and
$\triangle{A_p B_p C_p}$ are  polars:
the vertices of the former are the poles of the later and vice-verse;

ii) $A_p$ is on circle $\mathcal{C}_p.$   
\label{lemma:poristic2052}
\end{proposition}
\begin{proof} 
First note that, by construction and by the fundamental theorem on pole-polars, triangles $\mathcal{T}$
and $\mathcal{T}_p$ are mutually polars

Let $\gamma_D$ and $\Gamma_D$ be, respectively, the duals 
of circles $\mathcal{C}_p$ and $\mathcal{C}.$ 

 Since $\mathcal{C}_p$ is the circumcircle of $\mathcal{T}_p,$ then its dual $\gamma_D$ is
the i-conic focused in $D$ of triangle $\mathcal{T}.$ 
 
Similarly,  $\Gamma_D$, the dual of $\mathcal{C}$ is 
the i-conic focused in $D$ of  $\mathcal{T}_p$.

By Poncelet's porism,  $\big(\mathcal{C},\gamma_D\big)$ form a poristic pair for $n=3,$  since triangle
$\mathcal{T}$ is 
inscribed into the former and circumscribed to the later. Similarly, since by hypothesys 
$\mathcal{T}_p$ is inscribed in $\mathcal{C}_p$ and circumscribed to $\Gamma_D,$
the  $\big(\mathcal{C}_p,\Gamma_D\big)$ form a poristic pair for $n=3.$ 

Thus, if 
$A\in \mathcal{C}$ is any point and $B_pC_p$ is its polar 
$(B_p,C_P\in \mathcal{C}_p)$ then 
 $B_p C_p$ is a tangent to the conic $\Gamma_p,$ 
since the later is, by hypothesis,  the  dual of $\mathcal{C}$, hence the envelope of the polars of (all) points in  $\mathcal{C}$.

Similarly, since by construction, $B_p$ and $C_p$ are on the circle $\mathcal{C}_p,$
their polars  are the tangents from $A$ to   $\gamma_D,$ the dual  of $\mathcal{C}_p,$  again by the definition of a dual curve and by 
the fundamental theorem on pole-polars.

So, if these polars intercept circle $\mathcal{C}$ in $A,C$ and $A,B$ respectively, then $AB$ and $AC$ tangent 
$\gamma_D$ and $A,B,C$ are all on $\mathcal{C}.$ 
Thus, 
by Poncelet's porism, the line $BC$  necessarily  tangent
the  $\gamma_D.$  
Therefore, its pole, which is, by hypothesis,
the point $A_p$ is on the dual of $\Gamma_p,$ the circle  
$\mathcal{C}_p.$

Thus, triangles $\triangle{ABC}$
and $\triangle{A_p B_p C_p}$ are mutually polar
and their vertices are located on $\mathcal{C}$ and $\mathcal{C}_p$ respectively.
\end{proof}

\begin{figure}

\centering
\includegraphics[trim=0 0 0 0,clip,width=1.0\textwidth]{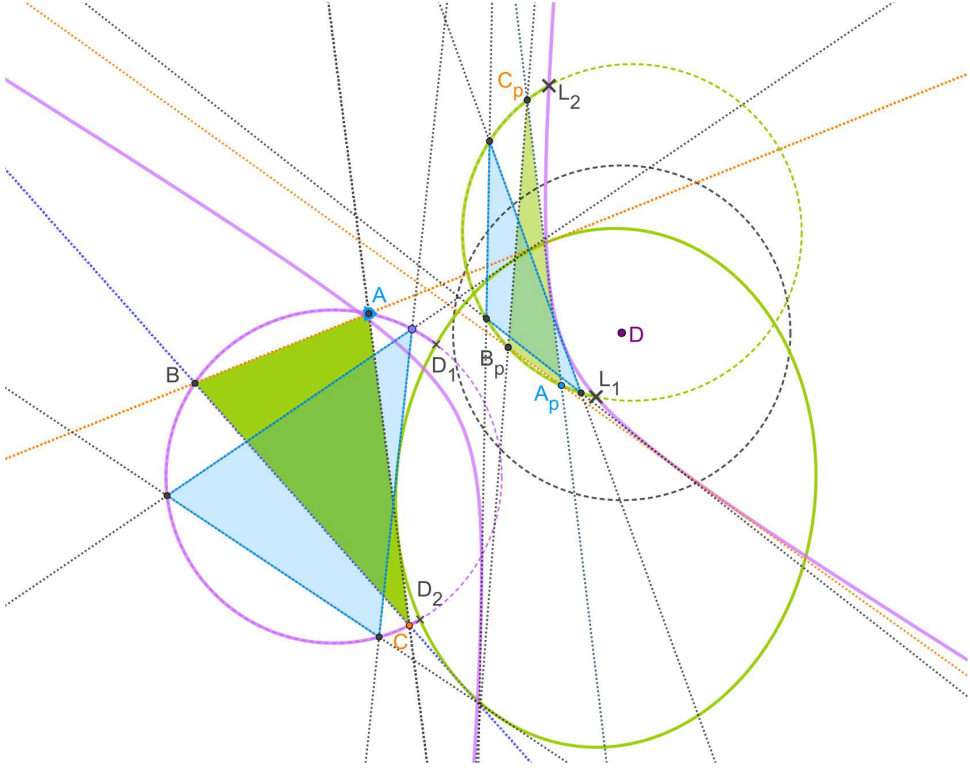}
\caption{ $\mathcal{C}$ is the circumcircle of 
$\triangle{ABC}$ (violet circle) and 
 $\mathcal{C}_p$ (green circle), the circumcircle of the polar triangle $\triangle{A_p B_p C_p}.$
 $\mathcal{C}_p$ is fixed, does not depend on point $A$ on $\mathcal{C}.$
 The  dual of $\mathcal{C}$  
w.r. to inversion circle centred in $D$
is  $\Gamma_D$ (violet hyperbola) and the dual of 
$\mathcal{C}_p$ is 
$\gamma_D$ (green ellipse). 
Then $\gamma_D$ is the i-conic of $\triangle{ABC}$
and $\Gamma_D$ is the i-conic of $\triangle{A_p B_p C_p}.$
 There are three poristic pairs:
i) $\big(\mathcal{C},\gamma_D);$ the fertile arc
$D_1D_2$ of $\mathcal{C}$ lies outside the i-conic; ii) $\big(\mathcal{C}_p,\Gamma_D);$
the fertile arc $L_1L_2$ of $\mathcal{C}_p$ lies outside the i-conic $\Gamma_D;$
iii) the polar-poristic 
$\big(\mathcal{C},\mathcal{C}_p\big);$
(shown $\triangle{ABC}$ and its polar 
$A_p B_p C_p)$.
The poristic pairs $\big(\mathcal{C},\gamma_D)$ and $\big(\mathcal{C}_p,\Gamma_D)$ are dual images of each other.
}
\label{fig:2052_2}

\end{figure}

\begin{figure}
\centering
\includegraphics[trim=10 0 0 0,clip,width=1.0\textwidth]{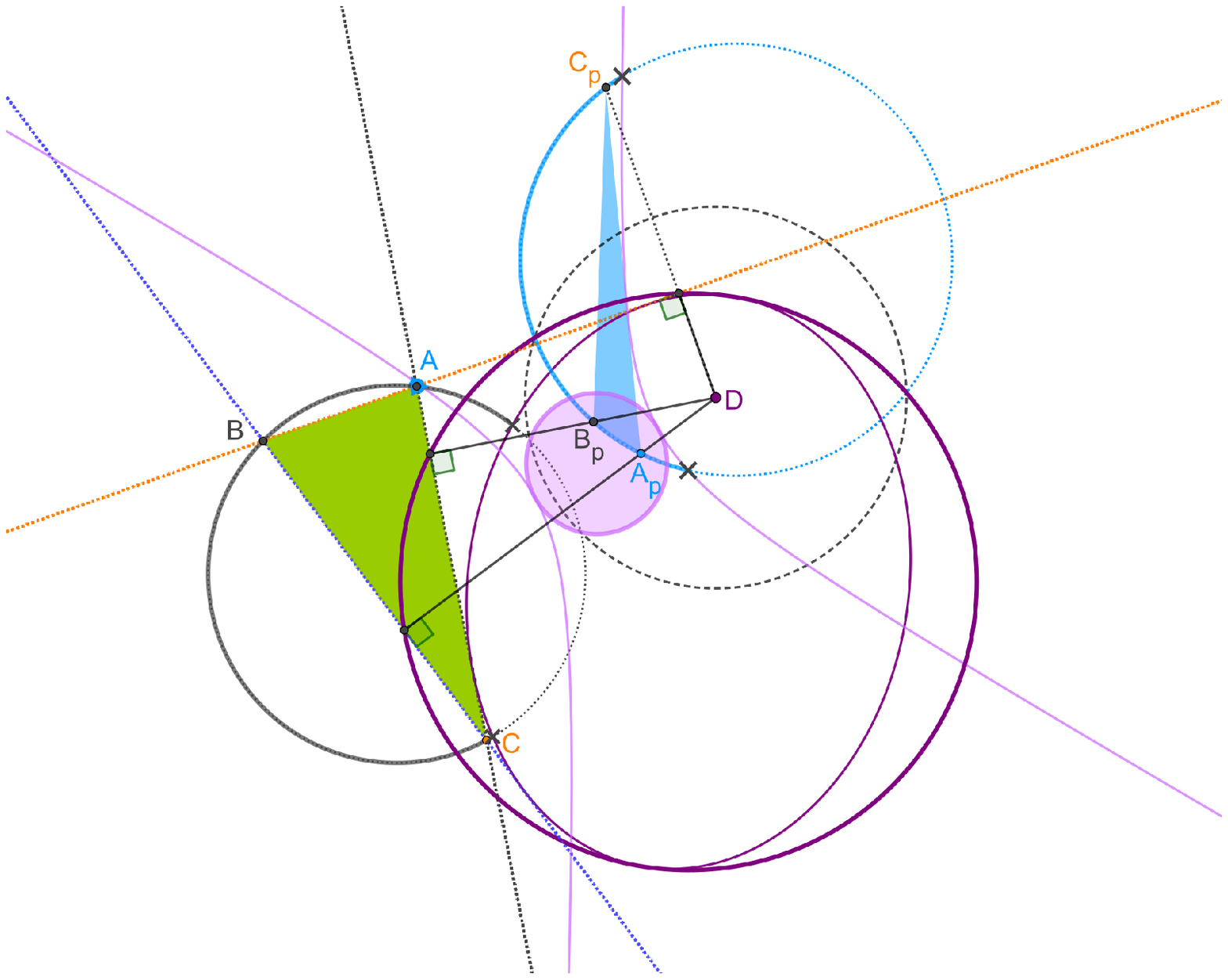}
\caption{
$\mathcal{E}_D$ the pedal-circle of 
$\triangle{ABC}$
(purple) coincides with the inverse of $\mathcal{C}_p$, the circumcircle of the polar triangle
$\triangle{A_p B_p C_p}$ (blue).Similarly,  the pedal of the polar triangle (solid violet) is the inverse of $\mathcal{C}$, the circumcircle of the original.
The perpendiculars from $D$ to the sides of $\triangle{ABC}$ intercepts them  on points located on pedal-circle $\mathcal{E}_D$ and
pass through the vertices of 
$\triangle {A_p B_p C_p}.$ 
}
\label{fig:pedal-conic}
\end{figure}

The polar-porism revealed in Proposition
~\ref{lemma:poristic2052} implicitly prove the following.

\begin{corollary} 

i) The poles $A_p,B_p,C_p$ of the sides of  triangles $\triangle{ABC}$ (which are) inscribed into a circle $\mathcal{C}$ describes a given  circle,  $\mathcal{C}_p$ if and only if the sides of (those)
$\triangle{ABC}$ tangents a  (one and the same) conic, $\gamma_D$.
  
 ii)  In this case, the sides of the respective polar triangle
 $\triangle{A_p B_p C_p}$ tangents a caustic conic
 $\Gamma_p.$
 \label{cor:caustic_conic}   
\end{corollary}
 The following result relates pedal-circle of the original (triangle) with the circumcircle of its  polar triangle.

 \begin{proposition}
 The pedal-circle of a triangle 
  is the inverse of the circumcircle of its polar triangle, w.r. to an inversion circle centered into the pedal-point.
 \end{proposition}
 
 \begin{proof}
 Refer to figure~\ref{fig:pedal-conic}.
 The i-conic $\gamma_D$ inscribed in $\triangle{ABC}$ is the negative-pedal of  
 pedal-circle of $\triangle{ABC};$
 as such, $\gamma_D$ is the dual of the inverse of 
 $\mathcal{E}_D:$
 $$\gamma_D=\mathcal{R}\big[ \mathcal{I}(\mathcal{E}_D)\big];$$
 performing the dual and using the fact that  the duality is an involution, we get
$$\mathcal{R}(\gamma_D)=\mathcal{I}(\mathcal{E}_D).$$ 
 The dual of $\gamma_D$, (a conic focused in $D$), w.r. to $\mathcal{I}$, (an inversion circle centred in $D$), is a circle.
 This circle is the  loci of the poles of the tangents at $\gamma_D.$ Further, since $\gamma_D$ is, by construction, the i-conic of $\triangle{ABC},$
  the dual of $\gamma_D$ also is the circumcircle of the polar triangle $\triangle{A_p B_p C_p},$ which is what we needed to proof.
  \end{proof}

  \begin{corollary}
  The lines  joining  the pedal-point $D$ with the vertices
  of the polar triangle $\mathcal{T}_p$, intercepts the sides of triangle $\mathcal{T}$ in points located on the pedal-circle.
  \end{corollary}
 \begin{proof} Refer to figure ~\ref{fig:pedal-conic}.
 The i-conics
$\gamma_D$   and $\Gamma_D$ 
inscribed in $\triangle{ABC}$ and  
$\triangle{A_p B_p C_p},$    respectively, are  the negative-pedals of their pedal-circles.
Since the negative-pedal is the reciprocal of the inverse,  their  pedal-circles
$\mathcal{E}_D$ and $\mathcal{E}'_D$ are, respectively, the   inverses of $\mathcal{C}_p$ and $\mathcal{C}$.
The perpendiculars from pedal-point $D$ to the sides of $\triangle{ABC}$ intercepts it  on points located on pedal-circle $\mathcal{E}_D$ and
pass through the vertices of 
$\triangle {A_p B_p C_p}.$ 
The same for $\triangle A_p B_p C_p.$
\label{lemma:pes-das-perpendiculares}
\end{proof}
 
\begin{corollary}
     The poles
of the sides of any triangle are located  on the inverse of its pedal-circle.
\end{corollary}

We are almost ready to give 
a recipe for the construction of all triangles sharing the circumcircle and pedal-circle, using polar triangles.

We only  have to pay attention to a phenomenon which may occur, whenever the pedal point $D$ lies outside 
$\mathcal{T}.$

Refer to figure ~\ref{fig:2052_2}.
\begin{lemma}(the lemma of the infertile arcs)
Let $\bf{\mathcal{T}}$ a triangle,
$\mathcal{C}$ its circumcircle, 
$D$ a pedal point neither on its sides,
nor on its circumcircle 
and $\mathcal{I}$ an inversion circle centred in $D$.    
    Let $\gamma_D$ the i-conic of $\bf{\mathcal{T}}$ focused in $D.$ 
    There are two cases.
    
    i) If $D$ is inside $\mathcal{T}$, then
    $\gamma_D$ is inside the circumcircle and tangents internally the sides of triangle $\mathcal{T}.$ In this case, 
    for all $A$ in $\mathcal{C}$ the polar of $A$ intercepts
    $\mathcal {T}_p.$
    
    ii) If $D$ is outside $\triangle{T}$, then 
    the i-conic $\gamma_D$ tangents externally $\mathcal{T}$ and intercepts  $\mathcal{C}$ in two distinct points $D_1$ and $D_2;$ let $D_1D_2$ the arc of $\mathcal{C}$
    that contain no vertex of $\mathcal{T};$
    then a polar of a point $A$ in $\mathcal{C}$ intercepts 
    $\mathcal{C}_p$ if and only if point $D$ 
    is located outside $\mathcal{T}.$
    \label{lemma:infertil_polar_arcs}    
\end{lemma}
\begin{proof}
The proof uses the fact that 
$\mathcal{C}_p,$ the circumcircle of the polar triangle is  the dual of $\gamma_D$. Thus, the polar of points located on on $\gamma_D$ are the tangents to $\mathcal{C}_p,$ while, by a continuity argument, the polar of points located into interior of $\Gamma$ 
\footnote{by internal points of a conic we mean points located into the same (conex) component of the plane as those containing one of its foci}
are external to $\mathcal{C}_p.$ Therefore
no polar of points that lie on the arc $D_1D_2$ of $\mathcal{C}$ intercepts $\mathcal{C}_p.$
\end{proof}

Finally, a recipe.
Refer to figure ~\ref{fig:2052_2}.
\begin{alg} 
Let $\mathcal{C}$ be the circumcircle and $\mathcal{E}_D$ be
the  pedal-circle of a triangle $\mathcal{T}$ and  let $\mathcal{E}_D'$  the inverse of $\mathcal{E}_D$, w.r. to inversion circle $\mathcal{I}$ centred in $D.$

Let $A$ be any point on $\mathcal{C}$  and let $a$ be 
its polar; if $a$ is external to $\mathcal{E}_D'$, then there is no triangle with vertex in $A$ sharing the same  circumcircle and pedal-circle with $\mathcal{T}$. Otherwise, let the polar of $A$ 
intercept   $\mathcal{E}_D'$  in $B_p,C_p;$  the polars
of $B_p$ and $C_p$ intercepts (again) $\mathcal{C}$ in  $C$ and $A$ and $B,$ respectively. Then, the
pedal-circle of $\triangle {ABC}$ thus construct is  $\mathcal{E}_D$.
All triangles that are inscribed in $\mathcal{C}$ and which shares the same pedal-circle are obtainable in this manner.
\end{alg}
\begin{remark} As a mater a fact, 
 when perform this construction, we obtain (concomitantly) two  systems of  triangles, one, poristic 
 w.r. to $\big(\mathcal{C},\gamma_D\big),$  and the other one,  poristic w.r. to  $\big(\mathcal{C}_p,\Gamma_D\big)$ and both respectively sharing 
the same circumcircle and  pedal-circle. Since we are interested in the former, the later was discarded.
\end{remark}

\section{A negative-pedal porism}

\begin{definition} (the negative-pedal triangle and the
negative-pedal circle)
Let $\mathcal{T}$ a triangle, $\mathcal{C}$ its circumcircle
and $D$ a pedal point neither  on its sides, nor on its circumcircle.
The negative-pedal triangle (or negative-pedal triangle, or the negative-pedal triangle)
 of  $\mathcal{T}$  w.r. to pedal point $D$ 
 is a triangle $\mathcal{T}',$
whose sides  $a',b'$ and $c'$ are the perpendiculars through the vertices of $\mathcal{T}$ to lines
that join the pedal-point $D$ to the vertices of $\mathcal{T}$.

The  circumcircle of the negative-pedal triangle, denoted by $\mathcal{C}_D$, is   the negative-pedal (or negative-pedal or negative-pedal) circle. 
\label{negative-pedal}
\end{definition}

Many classic circles (or triangles) may be looked upon as    negative-pedal circles (or triangles)  w.r. to pedal points that are notorious centres of the reference triangle.
The negative-pedal triangle the circumcenter  is the tangential triangle. The negative-pedal triangle of the orthocenter is the anti-complementary triangle.
Negative-pedal triangle is not defined
for the points that belong to the sides of the given triangle.

Negative-pedal circle, pedal-circle and circumcircle are closely related.

\begin{figure}
\centering
\includegraphics[trim=400 0 50 0,clip,width=1.0\textwidth]{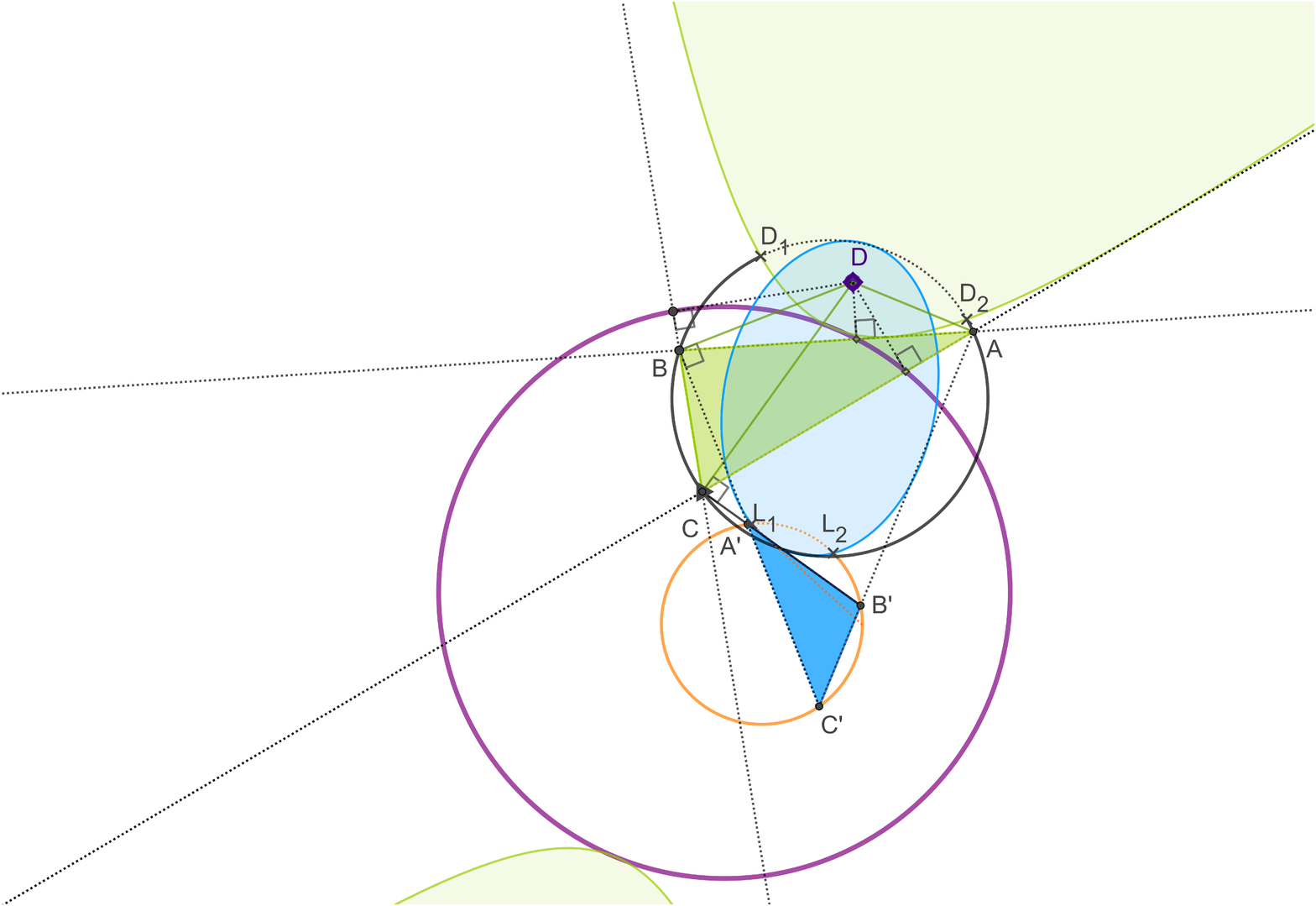}
\caption{The projections of pedal-point $D$ on the sides of $\triangle ABC$ are on pedal-circle 
$\mathcal{E}_D$ (purple) iff lines $AB,$ $AC,$ $BC$   tangent $\gamma_D$ the negative-pedal of circumcircle $\mathcal{C}$ (solid green hyperbola).
The vertices $A,B,C$ lie outside the hyperbola, hence  the  arc $D_1D_2$ 
situated inside the i-conic $\gamma_D$ is infertile.
$\triangle{ABC}$ preserves pedal-circle  iff 
the sides of pedal $\triangle{A' B' C'}$ (solid blue) tangent the negative-pedal of the circumcircle $\mathcal{C}$, the in-ellipse $\Gamma_D$ (blue).  $\mathcal{E}_D$ is fixed iff $\mathcal{C_D}$, the circumcircle of negative-pedal triangle $\triangle{A'B'C'}$ is fixed.
}
\label{fig:2049}
\end{figure}

Refer to figure ~\ref{fig:2049}
\begin{proposition}
The pedal-circle of $\mathcal{T}'$, the  negative-pedal triangle of $\mathcal{T}$ is  $\mathcal{C},$ the circumcircle of $\mathcal{T}$. 
\label{lemma:euler_pedal_circumcircle}
\end{proposition}

The proof of this fact is straightforward and we omit it.
   The results above update Proposition
~\ref{prop:poristica}.

  While 
the poristic tie between  circumcircle and 
pedal-circle requires the mediation of an  i-conic, there exists a straightforward poristic bound between the negative-pedal circle and circumcircle.

\begin{figure}
\centering
\includegraphics[trim=300 0 300 0,clip,width=1.0\textwidth]{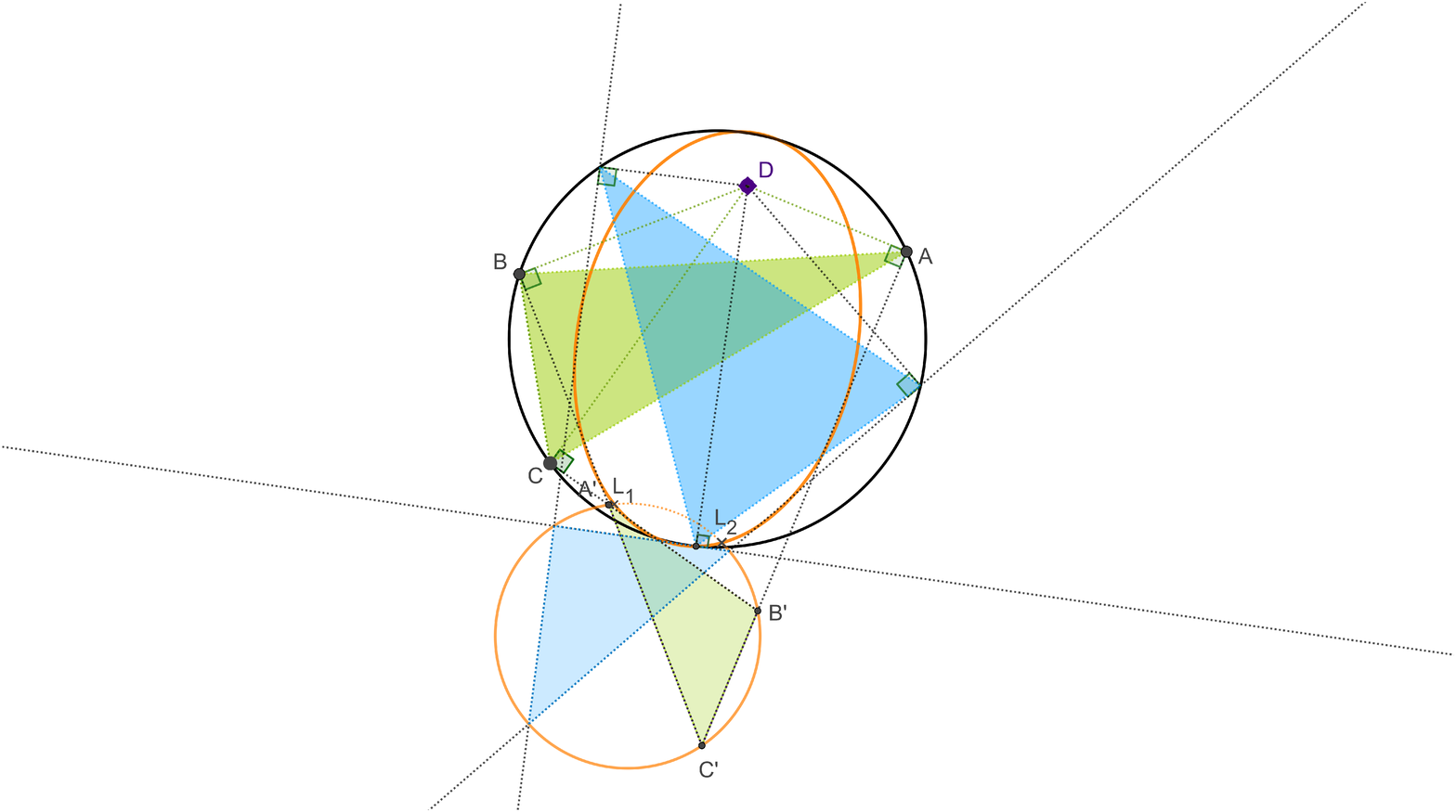}
\caption{The negative-pedal triangle 
 of $\triangle{ABC}$ w.r. to pedal-point $D$ is  
 $\triangle{A' B' C'},$
whose sides are the  perpendiculars in $A$, $B$, and  $C$ to $DA,DB$ and $DC$ respectively. Its circumcircle,  $\mathcal{C}_D$  (orange circle) is the  negative-pedal circle of $\triangle{ABC}$
w.r. to $D$. 
 $\Gamma_D,$ the i-conic of $\triangle{A' B' C'}$ focused in $D$ (orange elipse) is the negative-pedal of  circumcircle $\mathcal{C}$.
 $\big(\mathcal{C}_D,{\Gamma}_D\big)$ form a poristic pair for $n=3.$ By Poncelet's porism, 
two triangles inscribed in $\mathcal{C}$ 
(blue and green) share the same negative-pedal circle $\mathcal{C}_D$
if and only if the sides of their negative-pedal triangles tangents  the same conic  ${\Gamma}_D$. The arc $L_1 L_2$ of $\mathcal{C}_D$ located inside the in-ellipse, cannot contain  vertices
of $\triangle{A' B' C'}$ (there is no  tangent from 
on $L_1 L_2$ to $\Gamma_D$).
Circle
$\mathcal{C}$ is pedal-circle of $\triangle{A'B'C'}$ w.r. to pedal-point $D.$
}
\label{fig:2043}
\end{figure}

\begin{figure}
\centering
\includegraphics[trim=300 0 400 0,clip,width=1.0\textwidth]{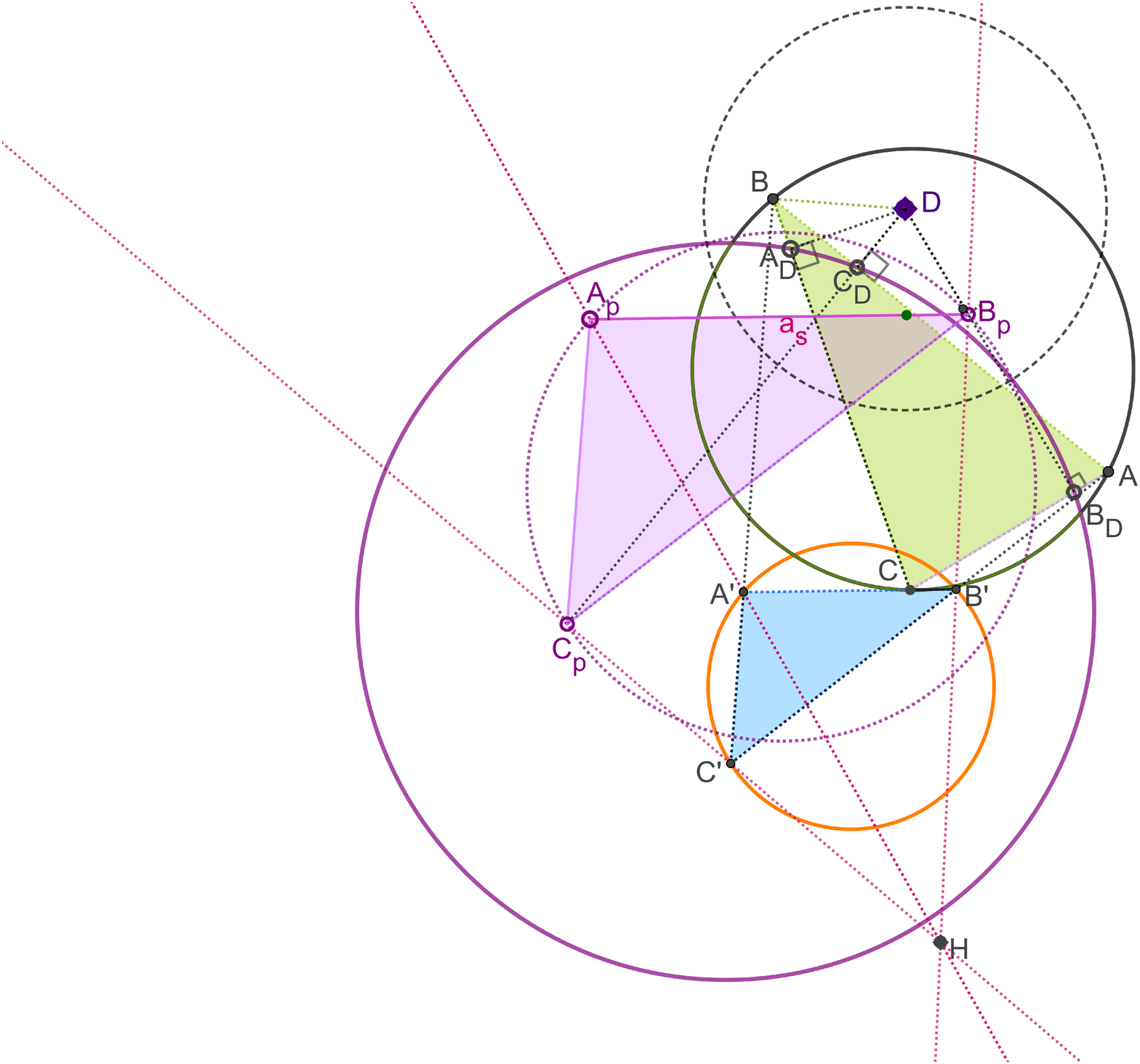}
\caption{$A_p$, $B_p$, $C_p$ are, respectively, the poles of $BC$, $AC$, and $AB$ w.r. to  $\mathcal{I}$ 
(doted black).$\mathcal{E}'_D$ (doted purple), 
the circumcircle of $\triangle{A_p B_p C_p}$
is the inverse of the pedal-circle $\mathcal{E}_D$ (purple). The negative-pedal 
$\triangle{A' B' C'}$ (blue) and $\triangle A_p B_p C_p$  are homothetic and so are their circumcircles
$\mathcal{C}_D$ and $\mathcal{E}'_D.$
Therefore
$\mathcal{C}_D$ is fixed iff $\mathcal{E}'_D$ is fixed iff 
$\mathcal{E}_D$ is fixed.
}
\label{fig:2047}
\end{figure}

\begin{proposition}
 (A negative-pedal porism).
Let $\mathcal{T}$ a triangle, $\mathcal{T}_D$ its negative-pedal triangle and $\mathcal{C}$, $\mathcal{C}_D$, their circumcircles. 

Then $\big(\mathcal{C}, \mathcal{C}_D\big)$ form a 
negative-pedal poristic pair of circles, in the following sense:
 from any point  $A$  on  $\mathcal{C}$ let
the perpendicular through $A$ to $DA$
intercept $\mathcal{C}_D$ in two distinct points $B'$ and $C'$.
The circles of diameters $[B'D]$ and $[C'D]$ intercept (again) circle $\mathcal{C}$ in $C$ and $B$, respectively.
Finally, let $A'$ the intersection of $BC'$ and $CB'$.
Then:

i) $A'$ is on circle $\mathcal{C}_D$;

ii) $\triangle{A' B' C'}$ share with $\mathcal{T'}$ the same circumcircle and pedal-circle $\mathcal{C}.$

If the perpendicular through $A$ to $DA$ does not 
intercept $\mathcal{C}_D$ in two distinct points, there is no such triangle with a vertex in $A.$ 
This can  only occur iff $A$ is located on an infertile arc $D_1D_2$ of $\mathcal{C}$ located   inside $\gamma_D$, the i-conic of $\triangle{ABC}.$
(see figure ~\ref{fig:2049}).
\label{proposition:negative_pedal_porism}
\end{proposition}

    \begin{proof}
    Refer to figure ~\ref{fig:2043}.
    Let $A'$  the intersection of $C' B$ with the circle $\mathcal{C}_D.$ Then 
    a (classic) poristic pair is 
    $(\mathcal{C}_D, \Gamma_D)$
 formed by the negative-pedal circle $\mathcal{C_D}$ 
  and the negative-pedal of the circumcircle $\mathcal{C}$, the in-ellipse $\Gamma_D$.
The construction performed  above is equivalent with Poncelet porism for this pair of conics
$(\mathcal{C}_D, \Gamma_D)$
taking as an initial point $A.$  We left the details to the reader.
\end{proof}
\begin{corollary} Two triangles inscribed into the same circle, 
  share the negative-pedal circle if and only if (they) share the same pedal-circle. 
\end{corollary}

The negative-pedal porism above allow a simple construction of all
triangles sharing  circumcircle and  negative-pedal circle. Refer to figure \ref{fig:2043}.
\begin{alg} 
Let $\mathcal{T}$ a triangle, $\mathcal{T}_D$ its negative-pedal triangle and $\mathcal{C}$, $\mathcal{C}_D$, their circumcircles. 
Let $A$  a point in $\mathcal{C}.$ Let 
  $a'$ the perpendicular in $A$ to $DA$, and
 let $B',C'$  the intersection of $a'$
and  $\mathcal{C}'_D.$
The circles of diameters $[B'D]$ and $[C'D]$
intercept (again) the circumcircle $\mathcal{C}$ in $C$ and $B,$
respectively.
Then 
 
I i) the negative-pedal triangle  of $\triangle{ABC}$ is
 $\triangle{A' B' C'}$; equivalently,
$\triangle{A B C}$ is 
the pedal triangle  of $\triangle{A' B' C'};$
 
I ii) the circumcircle of $\triangle{A' B' C'}$ is 
$\mathcal{C}'_D;$ equivalently, the negative-pedal circle of $\triangle{ABC}$ is $\mathcal{C}'_D;$
the  pedal-circle  of $\triangle{A'B'C'}$
is  $\mathcal{C}$.

II)
If any of these intersections is empty,  the process stops. This happens if the initial point $A$ was located on an "infertile" arc $D_1D_2$ of $\mathcal{C}$ (see figure~\ref{fig:2049}).
\end{alg}

Finally, let us show how these two circles relate.

\begin{proposition} The negative-pedal circle and the inverse of the pedal-circle of a triangle are homothetic. Therefore,
two  triangles inscribed in $\mathcal{C}$ share the pedal-circle $\mathcal{E}_D$ if and only if they share the negative-pedal circle $\mathcal{C}_D$.
\label{prop:pedal_negative-pedal}
\end{proposition}

\begin{proof} Referring to figure ~\ref{fig:2047}.
The poles of  $BC$, $AC$, and $AB$ w.r. to an inversion circle $\mathcal{I}$ centered in  $D$,
are, respectively, the inverses of the feet  of $D$ on the sides of $\triangle{ABC};$ denote them by $A_p,B_p$ and $C_p.$ 
Therefore, the lines  $DC,$ $DB$ and $DA$ are perpendicular on $A_pB_p,$
$A_pC_p$ and $B_pC_p$ respectively.

On the other hand, by the definition of a negative-pedal triangle,
$DC,$ $DB$ and $DA$
are also perpendicular on $A'B'$, $A'C'$ and $B'C'$ respectively.
Therefore, $\triangle{A' B' C'}$ and $\triangle A_p B_p C_p$ are homothetic.
 Their homothety center, $H$ is also the homothety center of their circumcircles.
  Thus
$\mathcal{C}_D$, the circumcircle of the negative-pedal triangle is fixed if and only if $\mathcal{E}'_D$, the circumcircle of the polar is fixed. Since the 
later is the inverse of the Euler circle, it is fixed 
if and only if  $\mathcal{E}_D$ is fixed.
\end{proof}

\section{References}
[A] Akopyan, A.,Zaslavsky, A., 
\emph{Geometry of Connics}, 2
	Amer. Math. Soc., 2007.
	
[Am1]	Ameseder, A., 
\emph{Negative Fusspunktcurven der Kegelschnitte} Archiv Math. u. Phys. 64, 170-176, 1879.

[Am2] Ameseder, A. 
\emph{Theorie der negativen Fusspunktencurven} Archiv Math. u. Phys. 64, 164-169, 1879.
	
[BP] Brianchon, Poncelet,J.V.
\emph{Recherches sur la détermination d’une hyperbole équilatère, au moyen de quatre conditions données}
Annales de Mathématiques pures et appliquées, tome 11 (1820-1821), p. 205-220.

[Ch1] Chasles, M.,
\emph{Théorèmes sur les sections coniques confocales}
Annales de Mathématiques pures et appliquées, tome 18 (1827-1828), p. 269-276.

[Ch] M. Chasles, Traité des sections coniques, Gauthier-Villars, Paris, 1885.


[GRSH] Garcia, R., Reznik, D., Stachel, H., Helman, M.
\emph{A family of constant-areas deltoid
associated with the ellipse.} arXiv. arxiv.org/abs/2006.13166. 1, 2, 11
 (2020).
 
[G]  Gheorghe, L. G. \emph{When Euler (circle) meets Poncelet (porism)}, preprint (2020)

[GSO]   Glaeser,G., Stachel, H., Odehnal, B., \emph{The universe of Conics}, Springer Specktrum, Springer-Verlang Berlin Heidelberg 2016.

[La] Lawrence, J. D. A Catalog of Special Plane Curves. New York: Dover, pp. 46-49, 1972.

[Lo] Lockwood, E. H. A Book of Curves. Cambridge, England: Cambridge University Press, pp. 156-159, 1967.

[OW] Ostermann, A., Wanner, G. (2012). Geometry by Its History. Springer Verlag. 3

[Ox] Oxman,V.
\emph{On the existence of triangles with given circumcircle, incircle and one aditional element}, Forum Geometricorum, 5 (2005) 165-171.

[Pa] Pamfilos, P
\emph{Triangles sharing their Euler circle and circumcircle},
Int. J. of Geometry, 
9 (2020), No 1, 5-24.

[PS] Smarandache, F. and Patrascu, I. 
\emph{
The geometry of the orthological triangles
} (2020). https://digitalrepository.unm.edu/math fsp/260

[P] Poncelet, J.V.,
\emph{Traité des propriétés projectives des figures}, Bachelier, 1822.

[S] Salmon, G.,\emph{ A treatise on conic sections}, Longman, Brown, Green, Longraus, 1855.

[W] Weisstein, E.\emph{ Mathworld. MathWorld–A Wolfram Web Resource. mathworld.wolfram.}
com.(2019) 

[We] Weaver, J. 
\emph{A system of triangles related to a poristic system  },
Amer. Math. Month, 31 (1924), 337-340. 


\bigskip
\bigskip
\bigskip

DEPARTAMENTO DE MATEMÁTICA

UNIVERSIDADE FEDERAL DE PERNAMBUCO

RECIFE, (PE) BRASIL

\textit{E-mail address}:

\texttt{liliana@dmat.ufpe.br}
\bigskip

\bigskip

\end{document}